\documentclass[a4paper,10pt]{article}%
\usepackage{amsmath}%
\usepackage{amsfonts}%
\usepackage{amssymb}%
\usepackage{amsthm}
\usepackage{graphicx}
\usepackage[latin1]{inputenc}
\usepackage[T1]{fontenc}
\usepackage[english,french]{babel}

\addtocounter{equation}{15}

\newtheorem{theorem}{Theorem}[section]

\newtheorem{lemma}[theorem]{Lemma}

\newtheorem{proposition}[theorem]{Proposition}

\begin{document}

\sloppy

\title{Actions of groups of homeomorphisms on one-manifolds}
\author{Emmanuel Militon \thanks{supported by the Fondation Mathématique Jacques Hadamard. Centre de mathématiques Laurent Schwartz. \'Ecole Polytechnique. 91128 Palaiseau cedex. e-mail: emmanuel.militon@math.polytechnique.fr} }
\date{\today}
\maketitle

\setlength{\parskip}{10pt}

\selectlanguage{english}
\begin{abstract}
In this article, we describe all the group morphisms from the group of compactly-supported homeomorphisms isotopic to the identity of a manifold to the group of homeomorphisms of the real line or of the circle.
\end{abstract}
\section{Introduction}

Fix a connected manifold $M$ (without boundary). For an integer $r \geq 0$, we denote by $\mathrm{Diff}^{r}(M)$ the group of $C^{r}$-diffeomorphisms of $M$. When $r=0$, this group will also be denoted by $\mathrm{Homeo}(M)$. For a homeomorphism $f$ of $M$, the \emph{support} of $f$ is the closure of the set:
$$ \left\{ x \in M, \ f(x) \neq x \right\}.$$
We denote by $\mathrm{Diff}_{0}^{r}(M)$ ($\mathrm{Homeo}_{0}(M)$ if $r=0$) the identity component of the group of compactly supported $C^{r}$-diffeomorphisms of $M$ (for the strong topology). If $r \neq \mathrm{dim}(M)+1$, these groups are simple by a well-known and difficult theorem (see \cite{Ban}, \cite{Bou}, \cite{Fis}, \cite{Mat1}, \cite{Mat2}).

In \cite{Ghy2}, \'Etienne Ghys asked whether the following statement was true: if $M$ and $N$ are two closed manifolds and if there exists a non-trivial morphism $\mathrm{Diff}_{0}^{\infty}(M) \rightarrow \mathrm{Diff}_{0}^{\infty}(N)$, then $\mathrm{dim}(M) \geq \mathrm{dim}(N)$. In \cite{Man}, Kathryn Mann proved the following theorem. Take a connected manifold $M$ of dimension greater than $1$ and a one-dimensional connected manifold $N$. Then any morphism $\mathrm{Diff}_{0}^{\infty}(M) \rightarrow \mathrm{Diff}_{0}^{\infty}(N)$ is trivial: she answers Ghys's question in the case where the manifold $N$ is one-dimensional. Mann also describes all the group morphisms $\mathrm{Diff}_{0}^{r}(M) \rightarrow \mathrm{Diff}_{0}^{r}(N)$ for $r \geq 3$ when $M$ as well as $N$ are one-dimensional. The techniques involved in the proofs of these theorems are Kopell's lemma (see \cite{Nav} Theorem 4.1.1) and Szekeres's theorem (see \cite{Nav} Theorem 4.1.11). These theorems are valid only for a regularity at least $C^{2}$. In this article, we prove similar results in the case of a $C^{0}$ regularity. The techniques used are different.

\begin{theorem} \label{manoncirc}
Let $M$ be a connected manifold of dimension greater than $2$ and let $N$ be a connected one-manifold. Then any group morphism $\mathrm{Homeo}_{0}(M) \rightarrow \mathrm{Homeo}(N)$ is trivial.
\end{theorem}

The case where the manifold $M$ is one-dimensional is also well-understood.

\begin{theorem} \label{roncirc}
Let $N$ be a connected one-manifold. For any group morphism $\varphi: \mathrm{Homeo}_{0}(\mathbb{R}) \rightarrow \mathrm{Homeo}(N)$, there exists a closed set $K \subset N$ such that:
\begin{enumerate}
\item The set $K$ is pointwise fixed under any homeomorphism in $\varphi(\mathrm{Homeo}_{0}(\mathbb{R}))$.
\item For any connected component $I$ of $N-K$, there exists a homeomorphism $h_{I}: \mathbb{R} \rightarrow I$ such that:
$$ \forall f \in \mathrm{Homeo}_{0}(\mathbb{R}), \ \varphi(f)_{|I}=h_{I}fh_{I}^{-1}.$$
\end{enumerate}
\end{theorem}

\noindent \textbf{Remark}: By a theorem by Matsumoto (see \cite{Mats} Theorem 5.3), every group morphism $\mathrm{Homeo}_{0}(\mathbb{S}^{1}) \rightarrow \mathrm{Homeo}_{0}(\mathbb{S}^{1})$ is a conjugation by a homeomorphism of the circle. Moreover, any group morphism $\mathrm{Homeo}_{0}(\mathbb{S}^{1}) \rightarrow \mathrm{Homeo}(\mathbb{R})$ is trivial. Indeed, as the group $\mathrm{Homeo}_{0}(\mathbb{S}^{1})$ is simple, such a group morphism is either one-to-one or trivial. However, the group $\mathrm{Homeo}_{0}(\mathbb{S}^{1})$ contains torsion elements whereas the group $\mathrm{Homeo}(\mathbb{R})$ does not: such a morphism cannot be one-to-one.

\section{Proofs of Theorems \ref{manoncirc} and \ref{roncirc}}

Fix an integer $d \geq 1$. For a point $p$ in $\mathbb{R}^{d}$, we denote by $G_{p}^{d}$  the group $\mathrm{Homeo}_{0}(\mathbb{R}^{d}- \left\{ p \right\})$. This group is seen as a subgroup of $\mathrm{Homeo}_{0}(\mathbb{R}^{d})$ consisting of homeomorphisms which pointwise fix a neighbourhood of the point $p$. We will call embedded $(d-1)$-dimensional ball of $\mathbb{R}^{d}$ the image of the closed unit ball of $\mathbb{R}^{d-1}= \mathbb{R}^{d-1} \times \left\{0 \right\} \subset \mathbb{R}^{d}$ under a homeomorphism of $\mathbb{R}^{d}$. For an embedded $(d-1)$-dimensional ball $D \subset \mathbb{R}^{d}$ (which is a single point if $d=1$), we denote by $H^{d}_{D}$ the subgroup of $\mathrm{Homeo}_{0}(\mathbb{R}^{d})$ consisting of homeomorphisms which pointwise fix a neighbourhood of the embedded ball $D$. Finally, if $G$ denotes a subgroup of $\mathrm{Homeo}(\mathbb{R}^{d})$, a point $p \in \mathbb{R}^{d}$ is said to be fixed under the group $G$ if it is fixed under all the elements of this group. We denote by $\mathrm{Fix}(G)$ the (closed) set of fixed points of $G$.

The theorems will be deduced from the following propositions. The two first propositions will be proved respectively in Sections 3 and 4.

\begin{proposition} \label{auplusunpoint}
Let $\varphi: \mathrm{Homeo}_{0}(\mathbb{R}^{d}) \rightarrow \mathrm{Homeo}(\mathbb{R})$ be a group morphism. Suppose that no point of the real line is fixed under the group $\varphi(\mathrm{Homeo}_{0}(\mathbb{R}^{d}))$. Then, for any embedded $(d-1)$-dimensional ball $D \subset \mathbb{R}^{d}$, the group $\varphi(H^{d}_{D})$ admits at most one fixed point.
\end{proposition}

\begin{proposition} \label{aumoinsunpoint}
Let $\varphi: \mathrm{Homeo}_{0}(\mathbb{R}^{d}) \rightarrow \mathrm{Homeo}(\mathbb{R})$ be a group morphism. Then, for any point $p$ in $\mathbb{R}^{d}$, the group $\varphi(G^{d}_{p})$ admits at least one fixed point.
\end{proposition}

\begin{proposition} \label{rdoncirc}
For any group morphism $\psi: \mathrm{Homeo}_{0}(\mathbb{R}^{d}) \rightarrow \mathrm{Homeo}(\mathbb{S}^{1})$, the group $\psi(\mathrm{Homeo}_{0}(\mathbb{R}^{d}))$ has a fixed point.
\end{proposition}

\begin{proof}[Proof of Proposition \ref{rdoncirc}]
Recall that the group $\mathrm{Homeo}_{0}(\mathbb{R}^{d})$ is infinite and simple and that the group $\mathrm{Homeo}(\mathbb{S}^{1})/ \mathrm{Homeo}_{0}(\mathbb{S}^{1})$ is isomorphic to $\mathbb{Z}/2 \mathbb{Z}$. Hence any morphism $\mathrm{Homeo}_{0}(\mathbb{R}^{d}) \rightarrow \mathrm{Homeo}(\mathbb{S}^{1})/ \mathrm{Homeo}_{0}(\mathbb{S}^{1})$ is trivial. Therefore, the image of a morphism $\mathrm{Homeo}_{0}(\mathbb{R}^{d}) \rightarrow \mathrm{Homeo}(\mathbb{S}^{1})$ is contained in $\mathrm{Homeo}_{0}(\mathbb{S}^{1})$.

For some background about the bounded cohomology of groups and the bounded Euler class of a group acting on a circle, see Section 6 in \cite{Ghy1}. By \cite{Mat3} and \cite{MM}:
$$H^{2}_{b}( \mathrm{Homeo}_{0}(\mathbb{R}^{d}), \mathbb{Z})= \left\{ 0 \right\}.$$
Therefore, the bounded Euler class of a morphism $\mathrm{Homeo}_{0}(\mathbb{R}^{d}) \rightarrow \mathrm{Homeo}_{0}(\mathbb{S}^{1})$ vanishes: this action has a fixed point.
\end{proof}

\begin{proof}[Proof of Theorem \ref{manoncirc}]
Let $d = \mathrm{dim}(M)$. The theorem will be deduced from the following lemma.

\begin{lemma} \label{rdonr}
Any group morphism $\mathrm{Homeo}_{0}(\mathbb{R}^{d}) \rightarrow \mathrm{Homeo}(\mathbb{R})$ is trivial.
\end{lemma}

Let us see why this lemma implies the theorem. Consider a morphism $\mathrm{Homeo}_{0}(M) \rightarrow \mathrm{Homeo}_{0}(N)$. Take an open set $U \subset M$ homeomorphic to $\mathbb{R}^{d}$ and let us denote by $\mathrm{Homeo}_{0}(U)$ the subgroup of $\mathrm{Homeo}_{0}(M)$ consisting of homeomorphisms supported in $U$. By Lemma \ref{rdonr} and Proposition \ref{rdoncirc}, the restriction of this morphism to the subgroup $\mathrm{Homeo}_{0}(U)$ is trivial. Moreover, as the group $\mathrm{Homeo}_{0}(M)$ is simple, such a group morphism is either one-to-one or trivial: it is necessarily trivial in this case.
\end{proof}

\begin{proof}[Proof of Lemma \ref{rdonr}]
Take a group morphism $\varphi: \mathrm{Homeo}_{0}(\mathbb{R}^{d}) \rightarrow \mathrm{Homeo}(\mathbb{R})$. Suppose by contradiction that this morphism is nontrivial. Replacing if necessary $\mathbb{R}$ with a connected component of the complement of the closed set $\mathrm{Fix}(\varphi(\mathrm{Homeo}_{0}(\mathbb{R}^{d})))$, we can suppose that the group $\varphi(\mathrm{Homeo}_{0}(\mathbb{R}^{d}))$ has no fixed points.

Let us prove that, for any points $p_{1} \neq p_{2}$ in $\mathbb{R}^{d}$:
$$ \mathrm{Fix}(\varphi(G_{p_{1}}^{d})) \cap \mathrm{Fix}(\varphi(G_{p_{2}}^{d}))= \emptyset.$$
The proof of this fact requires the following lemma.

\begin{lemma} \label{fragpoint}
Let $d' \geq 1$ be an integer. Let $p_{1} \neq p_{2}$ be two distinct points in $\mathbb{R}^{d'}$. Then, for any homeomorphism $f$ in $\mathrm{Homeo}_{0}(\mathbb{R}^{d'})$, there exist homeomorphisms $f_{1}, \ f_{3}$ in $G^{d'}_{p_{1}}$ and $f_{2}$ in $G^{d'}_{p_{2}}$  such that:
$$ f=f_{1}f_{2}f_{3}.$$
\end{lemma}

\begin{proof}
Take a homeomorphism $f$ in $\mathrm{Homeo}_{0}(\mathbb{R}^{d'})$. Let $f_{1}$ be a homeomorphism in $G^{d'}_{p_{1}}$ such that $f_{1}^{-1}$ sends the point $f(p_{1})$ to a point which lies in the same connected component of $\mathbb{R}^{d'}- \left\{ p_{2} \right\}$ as the point $p_{1}$. Let $f_{2}$ be a homeomorphism in $G^{d'}_{p_{2}}$ which is equal to $f_{1}^{-1}f$ in a neighbourhood of the point $p_{1}$. The existence of the homeomorphism $f_{2}$ is easy to prove when $d'=1$, is a consequence of the Schönflies Theorem when $d'=2$ and of the annulus theorem by Kirby and Quinn when $d' \geq 3$ (see \cite{Kir} and \cite{Qui}). Changing if necessary the homeomorphism $f_{2}$ into the composition of the homeomorphism $f_{2}$ with a homeomorphism supported in a small neighbourhood of the point $p_{1}$, the homeomorphism $f_{3}=f_{2}^{-1}f_{1}^{-1}f$ belongs to $G^{d'}_{p_{1}}$.
\end{proof}

Take two points $p_{1}$ and $p_{2}$ in $\mathbb{R}^{d}$. Suppose by contradiction that $\mathrm{Fix}(\varphi(G_{p_{1}}^{d})) \cap \mathrm{Fix}(\varphi(G^{d}_{p_{2}})) \neq \emptyset$. By Lemma \ref{fragpoint}, a point in this set is a fixed point of the group $\varphi(\mathrm{Homeo}_{0}(\mathbb{R}^{d}))$, a contradiction.

By Proposition \ref{aumoinsunpoint}, the sets $\mathrm{Fix}(\varphi(G^{d}_{p}))$, for $p \in \mathbb{R}^{d}$ are nonempty. We just saw that they are pairwise disjoint. Recall that, for any embedded $(d-1)$-dimensional ball $D$, the set $\mathrm{Fix}(\varphi(H_{D}^{d}))$ contains the union of the sets $\mathrm{Fix}(\varphi(G^{d}_{p}))$ over the points $p$ in the closed set $D$. Hence, this set has infinitely many points as $d \geq 2$, a contradiction with Proposition \ref{auplusunpoint}. 
\end{proof}

\begin{proof}[Proof of Theorem \ref{roncirc}]
Let $\varphi : \mathrm{Homeo}_{0}(\mathbb{R}) \rightarrow \mathrm{Homeo}(N)$ be a nontrivial group morphism. By Proposition \ref{rdoncirc}, we can suppose that the manifold $N$ is the real line $\mathbb{R}$. Replacing $\mathbb{R}$ with a connected component of the complement of the closed set $\mathrm{Fix}(\varphi(\mathrm{Homeo}_{0}(\mathbb{R})))$ if necessary, we can suppose that the group $\varphi(\mathrm{Homeo}_{0}(\mathbb{R}))$ has no fixed point. Recall that the group $\mathrm{Homeo}_{0}(\mathbb{R})$ is simple. Hence any morphism $\mathrm{Homeo}_{0}(\mathbb{R}) \rightarrow \mathbb{Z}/ 2 \mathbb{Z}$ is trivial. Thus, any element of the group $\varphi(\mathrm{Homeo}_{0}(\mathbb{R}))$ preserves the orientation of $\mathbb{R}$. 

By Propositions \ref{auplusunpoint} and \ref{aumoinsunpoint}, for any real number $x$, the group $\varphi(G_{x}^{1})$ has a unique fixed point $h(x)$. Take a homeomorphism $f$ in $\mathrm{Homeo}_{0}(\mathbb{R})$ which sends a point $x$ in $\mathbb{R}$ to a point $y$ in $\mathbb{R}$. Then $fG_{x}^{1}f^{-1}=G^{1}_{y}$ and, taking the image under $\varphi$, $\varphi(f) \varphi(G_{x}^{1}) \varphi(f)^{-1}=\varphi(G^{1}_{y})$. Hence $\varphi(f)(\mathrm{Fix}(\varphi(G^{1}_{x})))= \mathrm{Fix}(\varphi(G^{1}_{y}))$. Therefore, for any homeomorphism $f$ in $\mathrm{Homeo}_{0}(\mathbb{R})$, $\varphi(f)h=hf$.

Let us prove that the map $h$ is one-to-one. Suppose by contradiction that there exist real numbers $x \neq y$ such that $h(x)=h(y)$. The point $h(x)$ is fixed under the groups $\varphi(G^{1}_{x})$ and $\varphi(G^{1}_{y})$. However, the groups $G^{1}_{x}$ and $G^{1}_{y}$ generate the group $\mathrm{Homeo}_{0}(\mathbb{R})$ by Lemma \ref{fragpoint}. Therefore, the point $h(x)$ is fixed under the group $\varphi(\mathrm{Homeo}_{0}(\mathbb{R}))$, a contradiction.

Now we prove that the map $h$ is either strictly increasing or strictly decreasing. Fix two points $x_{0} < y_{0}$ of the real line. For any two points $x<y$ of the real line, let us consider a homeomorphism $f_{x,y}$ in $\mathrm{Homeo}_{0}(\mathbb{R})$ such that $f_{x,y}(x_{0})=x$ and $f_{x,y}(y_{0})=y$. As $\varphi(f_{x,y})h=hf_{x,y}$, the homeomorphism $\varphi(f_{x,y})$ sends the ordered pair $(h(x_{0}),h(y_{0}))$ to the ordered pair $(h(x),h(y))$. As the homeomorphism $\varphi(f_{x,y})$ is strictly increasing:
$$ h(x)<h(y)\Leftrightarrow h(x_{0}) < h(y_{0})$$
and
$$ h(x)>h(y)\Leftrightarrow h(x_{0}) > h(y_{0}).$$
Hence the map $h$ is either strictly increasing or strictly decreasing.

Now, it remains to prove that the map $h$ is onto to complete the proof. Suppose by contradiction that the map $h$ is not onto. Notice that the set $h(\mathbb{R})$ is preserved under the group $\varphi(\mathrm{Homeo}_{0}(\mathbb{R}))$. If this set had a lower bound or an upper bound, then the supremum of this set or the infimum of this set would provide a fixed point for the group $\varphi(\mathrm{Homeo}_{0}(\mathbb{R}))$, a contradiction. This set has neither upper bound nor lower bound. Let $C$ be a connected component of the complement of the set $h(\mathbb{R})$. To simplify the exposition of the proof, we suppose that the map $h$ is increasing. Let us denote by $x_{0}$ the supremum of the set of points $x$ such that the real number $h(x)$ is lower than any point in the interval $C$. Then the point $h(x_{0})$ is necessarily in the closure of $C$: otherwise, there would exist an interval in the complementary of $h(\mathbb{R})$ which strictly contains the interval C. We suppose for instance that the point $h(x_{0})$ is the supremum of the interval $C$. Choose, for each couple $(z_{1}, z_{2})$ of real numbers, a homeomorphism $g_{z_{1},z_{2}}$ in $\mathrm{Homeo}_{0}(\mathbb{R})$ which sends the point $z_{1}$ to the point $z_{2}$. Then the sets  $g_{x_{0},x}(C)$, for $x$ in $\mathbb{R}$, are pairwise disjoint: they are pairwise distinct as their suprema are pairwise distinct (the supremum of the set  $g_{x_{0},x}(C)$ is the point $h(x)$). Moreover, those sets do not contain any point of $h(\mathbb{R})$ and the infima of those sets are accumulated by points in $h(\mathbb{R})$. Hence, these sets are pairwise disjoint. Then the set $C$ has necessarily an empty interior as the topological space $\mathbb{R}$ is second-countable. Therefore $C = \left\{ h(x_{0}) \right\}$, which is not possible.
\end{proof}

\section{Proof of Proposition \ref{auplusunpoint}}

The proof of this proposition is similar to the proofs of Lemmas 3.6 and 3.7 in \cite{Mil}. We need the following lemma. The proof of this lemma is almost identical to the proof of Lemma \ref{fragpoint} and is omitted.

\begin{lemma} \label{fragball}
Take two disjoint embedded $(d-1)$-dimensional balls $D$ and $D'$ in $\mathbb{R}^{d}$. For any homeomorphism $f$ in $\mathrm{Homeo}_{0}(\mathbb{R}^{d})$, there exist homeomorphisms $h_{1}$, $h_{3}$ in $H^{d}_{D}$ and $h_{2}$ in $H^{d}_{D'}$ such that
$$ h= h_{1}h_{2}h_{3}.$$
\end{lemma}

For such an embedded $(d-1)$-dimensional ball $D$, let $F_{D}=\mathrm{Fix}(\varphi(H_{D}^{d}))$. 
Let us prove that these sets are pairwise homeomorphic. Take two embedded $(d-1)$-dimensional balls $D$ an $D'$ and take a homeomorphism $h$ in $\mathrm{Homeo}_{0}(\mathbb{R}^{d})$ which sends the set $D$ onto $D'$. Observe that $hH_{D}^{d}h^{-1}=H^{d}_{D'}$ and that $\varphi(h) \varphi(H_{D}^{d}) \varphi(h)^{-1}=\varphi(H^{d}_{D'})$. Therefore: $\varphi(h)(F_{D})=F_{D'}$.

In the case where these sets are all empty, there is nothing to prove. We suppose in what follows that they are not empty.

Given two disjoint embedded $(d-1)$-dimensional balls $D$ and $D'$, Lemma \ref{fragball} implies, as in the proof of Lemma \ref{rdonr}:
$$ F_{D} \cap F_{D'}= \emptyset.$$

\begin{lemma} \label{compl}
Fix an embedded $(d-1)$-dimensional ball $D_{0}$ of $\mathbb{R}^{d}$. Then any connected component $C$ of the complement of $F_{D_{0}}$ meets one of the sets $F_{D}$, where $D$ is an embedded $(d-1)$-dimensional ball disjoint from $D_{0}$.
\end{lemma}

\begin{proof}
Let $(a_{1},a_{2})$ be a connected component of the complement of $F_{D_{0}}$. It is possible that either $a_{1}= -\infty$ or $a_{2}=+ \infty$. Consider a homeomorphism $e: \mathbb{R}^{d-1} \times \mathbb{R} \rightarrow \mathbb{R}^{d}$ such that $e( B^{d-1} \times \left\{0 \right\})=D_{0}$, where $B^{d-1}$ denotes the unit closed ball in $\mathbb{R}^{d-1}$. For any real number $x$, let $D_{x}=e( B^{d-1} \times \left\{ x \right\})$. Given two real $x \neq y$, take a homeomorphism $\eta_{x,y}$ in $\mathrm{Homeo}_{0}(\mathbb{R})$ which sends the point $x$ to the point $y$. Consider a homeomorphism $h_{x,y}$ such that the following property is satisfied. The restriction of $eh_{x,y}e^{-1}$ to $ B^{d-1} \times \mathbb{R}$ is equal to the map:
$$ \begin{array}{rcl}
 B^{d-1} \times \mathbb{R}  & \rightarrow & \mathbb{R}^{d-1} \times \mathbb{R} \\
(p,z) & \mapsto &(p,\eta_{x,y}(z))
\end{array}
$$
Notice that, for any real numbers $x$ and $y$, $h_{x,y}(D_{x})=D_{y}$

Let us prove by contradiction that there exists a real number $x \neq 0$ such that $F_{D_{x}} \cap (a_{1},a_{2}) \neq \emptyset$. Suppose that, for any such embedded ball $D_{x}$, $F_{D_{x}} \cap (a_{1},a_{2})= \emptyset$. We claim that the open sets $\varphi(h_{0,x})((a_{1},a_{2}))$ are pairwise disjoint. It is not possible as there would be uncountably many pairwise disjoint open intervals in $\mathbb{R}$. 

Indeed, suppose by contradiction that there exists real numbers $x \neq y$ such that $\varphi(h_{0,x})((a_{1},a_{2})) \cap \varphi(h_{0,y})((a_{1},a_{2})) \neq \emptyset$. Notice that the homeomorphism $h_{0,x}^{-1} h_{0,y}$ and $h_{0,y}^{-1} h_{0,x}$ send respectively the set $D_{0}$ to sets of the form $D_{z}$ and $D_{z'}$, where $z,z' \in \mathbb{R}$. Hence, for $i=1,2$, the homeomorphisms $\varphi(h_{0,x}^{-1} h_{0,y})$ (respectively $\varphi(h_{0,y}^{-1} h_{0,x})$) sends the point $a_{i} \in F_{D_{0}}$ to a point in $F_{D_{z}}$ (respectively in $F_{D_{z'}}$). By hypothesis, these points do not belong to $(a_{1},a_{2})$. Therefore $$\varphi(h_{0,y}^{-1} h_{0,x})(a_{1},a_{2})=(a_{1},a_{2})$$
or
$$\varphi(h_{0,x})(a_{1},a_{2})=\varphi(h_{0,y})(a_{1},a_{2}).$$
But this last equality cannot hold as the real endpoints of the interval on the left-hand side belong to $F_{D_{x}}$ and the real endpoints point of the interval on the right-hand side belongs to $F_{D_{y}}$. Moreover, we saw that these two closed sets were disjoint, a contradiction.
\end{proof}

\begin{lemma}
Each set $F_{D}$ contains only one point.
\end{lemma}

\begin{proof}
Suppose that there exists an embedded $(d-1)$-dimensional ball $D$ such that the set $F_{D}$ contains two points $p_{1}<p_{2}$. By Lemma \ref{compl}, there exists an embedded $(d-1)$-dimensional ball $D'$ disjoint from $D$ such that the set $F_{D'}$ has a common point with the open interval $(p_{1},p_{2})$. Take a real number $r<p_{1}$. Then, for any homeomorphisms $g_{1}$ in $G_{D}$, $g_{2}$ in $G_{D'}$ and $g_{3}$ in $G_{D}$,
$$ \varphi(g_{1}) \circ \varphi(g_{2}) \circ \varphi(g_{3})(r) < p_{2}.$$
By Lemma \ref{fragball}, this implies that the following inclusion holds:
$$ \left\{ \varphi(g)(r), g \in \mathrm{Homeo}_{0}(\mathbb{R}^{d}) \right\} \subset (-\infty, p_{2}].$$
The supremum of the left-hand set provides a fixed point for the action $\varphi$, a contradiction.
\end{proof}

\section{Proof of Proposition \ref{aumoinsunpoint}}

This proof uses the following lemmas. For a subgroup $G$ of $\mathrm{Homeo}_{0}(\mathbb{R}^{d})$, we define the support $ \mathrm{Supp}(G)$ of $G$ as the closure of the set:
$$ \left\{ x \in \mathbb{R}^{d}, \ \exists g \in G, \ gx \neq x \right\}.$$
Let $\mathrm{Homeo}_{\mathbb{Z}}(\mathbb{R})= \left\{ f \in \mathrm{Homeo}(\mathbb{R}), \ \forall x \in \mathbb{R}, f(x+1)=f(x)+1 \right\}.$

To prove Proposition \ref{aumoinsunpoint}, we need the following lemmas.

\begin{lemma} \label{commhomeoz}
Let $G$ and $G'$ be subgroups of the group $\mathrm{Homeo}_{+}(\mathbb{R})$ of orientation-preserving homeomorphisms of the circle. Suppose that the following conditions are satisfied.
\begin{enumerate}
\item The groups $G$ and $G'$ are isomorphic to the group $\mathrm{Homeo}_{\mathbb{Z}}(\mathbb{R})$.
\item The subgroups $G$ and $G'$ of $\mathrm{Homeo}_{+}(\mathbb{R})$ commute: $\forall g \in G, g' \in G', \ gg'=g'g.$
\end{enumerate}
Then $\mathrm{Supp}(G) \subset \mathrm{Fix}(G')$ and $\mathrm{Supp}(G') \subset \mathrm{Fix}(G)$.
\end{lemma}

\begin{lemma} \label{isomhomeoz}
Take any nonempty open subset $U$ of $\mathbb{R}^{d}$. Then there exists a subgroup of $\mathrm{Homeo}_{0}(\mathbb{R}^{d})$ isomorphic to $\mathrm{Homeo}_{\mathbb{Z}}(\mathbb{R})$ which is supported in $U$.
\end{lemma}

Lemma \ref{commhomeoz} will be proved in the next section. We now provide a proof of Lemma \ref{isomhomeoz}.

\begin{proof}[Proof of Lemma \ref{isomhomeoz}]
Take a closed ball $B$ contained in $U$. Changing coordinates if necessary, we can suppose that $B$ is the unit closed ball in $\mathbb{R}^{d}$. Take an orientation-preserving homeomorphism $h: \mathbb{R} \rightarrow (-1,1)$. For any orientation-preserving homeomorphism $f: \mathbb{R} \rightarrow \mathbb{R}$, we define the homeomorphism $\lambda_{h}(f): \mathbb{R}^{d} \rightarrow \mathbb{R}^{d}$ in the following way.
\begin{enumerate}
\item The homeomorphism $\lambda_{h}(f)$ is equal to the identity outside the interior of the ball $B$.
\item For any $(x_{1},x') \in \mathbb{R} \times \mathbb{R}^{d-1} \cap \mathrm{int}(B)$:
$$ \lambda_{h}(f)(x_{1},x')=(\sqrt{1- \left\| x' \right\|^{2}} h \circ f \circ h^{-1}(\frac{x_{1}}{\sqrt{1- \left\| x' \right\|^{2}}}),x').$$
\end{enumerate}
The map $\lambda_{h}$ defines an embedding of the group $\mathrm{Homeo}_{+}(\mathbb{R})$ into the group $\mathrm{Homeo}_{0}(\mathbb{R}^{d})$. The image under $\lambda_{h}$ of the group $\mathrm{Homeo}_{\mathbb{Z}}(\mathbb{R})$ is a subgroup of $\mathrm{Homeo}_{0}(\mathbb{R}^{d})$ which is supported in $U$.
\end{proof}

Let us complete now the proof of Proposition \ref{aumoinsunpoint}.

\begin{proof}[Proof of Proposition \ref{aumoinsunpoint}]
Fix a point $p$ in $\mathbb{R}^{d}$. Take a closed ball $B \subset \mathbb{R}^{d}$ which is centered at $p$. Let $G^{d}_{B}$ be the subgroup of $G^{d}_{p}$ consisting of homeomorphisms which pointwise fix a neighbourhood of the ball $B$. Let us prove that $\mathrm{Fix}(G_{B}^{d}) \neq \emptyset$.

Take a subgroup $G_{1}$ of $\mathrm{Homeo}_{0}(\mathbb{R}^{d})$ which is isomorphic to $\mathrm{Homeo}_{\mathbb{Z}}(\mathbb{R})$ and supported in $B$. Such a subgroup exists by Lemma \ref{isomhomeoz}. This subgroup commutes with any subgroup $G_{2}$ of $\mathrm{Homeo}_{0}(\mathbb{R}^{d})$ which is isomorphic to $\mathrm{Homeo}_{\mathbb{Z}}(\mathbb{R})$ and supported outside $B$.

If the group $\varphi(\mathrm{Homeo}_{0}(\mathbb{R}^{d}))$ admits a fixed point, there is nothing to prove. Suppose that this group has no fixed point. As the group $\mathrm{Homeo}_{0}(\mathbb{R}^{d})$ is simple, the morphism $\varphi$ is one-to-one. Moreover, any morphism $\mathrm{Homeo}_{0}(\mathbb{R}^{d}) \rightarrow \mathbb{Z} / 2 \mathbb{Z}$ is trivial: the morphism $\varphi$ takes values in $\mathrm{Homeo}_{+}(\mathbb{R})$. Hence the subgroups $\varphi(G_{1})$ and $\varphi(G_{2})$ of $\mathrm{Homeo}(\mathbb{R})$ satisfy the hypothesis of Lemma \ref{commhomeoz}. By this lemma:
$$ \emptyset \neq \mathrm{Supp}(\varphi(G_{1})) \subset \mathrm{Fix}(\varphi(G_{2})).$$
We claim that the group $G^{d}_{B}$ is generated by the union of its subgroups isomorphic to $\mathrm{Homeo}_{\mathbb{Z}}(\mathbb{R})$. This claim implies that
$$ \emptyset \neq \mathrm{Supp}(\varphi(G_{1})) \subset \mathrm{Fix}(\varphi(G^{d}_{B})).$$
For $d \geq 2$, the claim is a direct consequence of the simplicity of the group $G^{d}_{B}$. In the case where $d=1$, denote by $[a,b]$ the compact interval $B$. The inclusions of the groups $\mathrm{Homeo}_{0}((-\infty,a))$ and $\mathrm{Homeo}_{0}((b,+\infty))$ induce an isomorphism $\mathrm{Homeo}_{0}((-\infty,a)) \times \mathrm{Homeo}_{0}((b,+\infty)) \rightarrow G^{d}_{B}$. The simplicity of each factor of this decomposition implies the claim.

Now, let us prove that the set $\mathrm{Fix}(\varphi(G^{d}_{B}))$ is compact. Suppose by contradiction that there exists a sequence $(a_{k})_{k \in \mathbb{N}}$ of real numbers in $\mathrm{Fix}(\varphi(G^{d}_{B}))$ which tends to $+\infty$ (if we suppose that it tends to $-\infty$, we obtain of course an analogous contradiction). Let us choose a closed ball $B' \subset \mathbb{R}^{d}$ which is disjoint from $B$. Observe that the subgroups $G^{d}_{B}$ and $G^{d}_{B'}$ are conjugate in $\mathrm{Homeo}_{0}(\mathbb{R}^{d})$ by a homeomorphism which sends the ball $B$ to the ball $B'$. Then the subgroups $\varphi(G^{d}_{B})$ and $\varphi(G^{d}_{B'})$ are conjugate in the group $\mathrm{Homeo}_{+}(\mathbb{R})$. Hence the sets $\mathrm{Fix}(\varphi(G^{d}_{B}))$ and $\mathrm{Fix}(\varphi(G^{d}_{B'}))$ are homeomorphic: there exists a sequence $(b_{k})_{k \in \mathbb{N}}$ of elements in $\mathrm{Fix}(\varphi(G^{d}_{B'}))$ which tends to $+\infty$. Take positive integers $n_{1}$, $n_{2}$ and $n_{3}$ such that $a_{n_{1}} < b_{n_{2}} < a_{n_{3}}$. Fix $x_{0}<a_{n_{1}}$. We notice then that for any homeomorphisms $g_{1} \in G^{d}_{B}$, $g_{2} \in G^{d}_{B'}$ and $g_{3} \in G^{d}_{B}$, the following inequality is satisfied:
$$\varphi(g_{1}) \varphi(g_{2}) \varphi(g_{3})(x_{0})<a_{n_{3}}.$$
However, any element $g$ in $\mathrm{Homeo}_{0}(\mathbb{R}^{d})$ can be written as a product
$$ g=g_{1}g_{2}g_{3},$$
where $g_{1}$ and $g_{3}$ belong to $G^{d}_{B}$ and $g_{2}$ belongs to $G^{d}_{B'}$. The proof of this fact is similar to that of Lemma \ref{fragpoint}. Therefore:
$$ \overline{\left\{ \varphi(g)(x_{0}), \ g \in \mathrm{Homeo}_{c}(\mathbb{R}) \right\}} \subset (-\infty,a_{n_{3}}].$$
The greatest element of the left-hand set is a fixed point of the image of $\varphi$: this is not possible as this image was supposed to have no fixed point.

Observe that the group $\varphi(G^{d}_{p})$ is the union of its subgroup of the form $\varphi(G^{d}_{B'})$, with $B'$ varying over the set $\mathcal{B}_{p}$ of closed balls centered at the point $p$. By compactness, the set
$$\mathrm{Fix}(\varphi(G^{d}_{p}))= \bigcap_{B' \in \mathcal{B}_{p}} \mathrm{Fix}(G^{d}_{B'})$$
is nonempty. Proposition \ref{aumoinsunpoint} is proved.
\end{proof}

\section{Proof of Lemma \ref{commhomeoz}}

We need the following lemmas. The first one will be proved afterwards.

\begin{lemma} \label{homeozsurr}
Let $\psi: \mathrm{Homeo}_{\mathbb{Z}}(\mathbb{R}) \rightarrow \mathrm{Homeo}_{+}(\mathbb{R})$ be a group morphism. Then there exists a closed set $F \subset \mathbb{R}$ such that:
\begin{enumerate}
\item The set $F$ is pointwise fixed under any element in $\mathrm{Homeo}_{\mathbb{Z}}(\mathbb{R})$.
\item For any connected component $K$ of the complement of $F$, there exists a homeomorphism $h_{K}: \mathbb{R} \rightarrow K$ such that:
$$ \forall f \in \mathrm{Homeo}_{\mathbb{Z}}(\mathbb{R}), \forall x \in K, \ \psi(f)(x)=h_{K}fh_{K}^{-1}.$$
\end{enumerate} 
\end{lemma}

\begin{lemma} \label{perfect}
Any group morphism $\mathrm{Homeo}_{\mathbb{Z}}(\mathbb{R}) \rightarrow \mathbb{Z}$ is trivial.
\end{lemma}

\begin{proof}[Proof of Lemma \ref{perfect}]
Actually, any element in $\mathrm{Homeo}_{\mathbb{Z}}(\mathbb{R})$ can be written as a product of commutators, \emph{i.e.} elements of the form $aba^{-1}b^{-1}$, with $a,\ b \in \mathrm{Homeo}_{\mathbb{Z}}(\mathbb{R})$. For an explicit construction of such a decomposition, see Section 2 in \cite{EHN}.
\end{proof}

Observe that the center of the group $\mathrm{Homeo}_{\mathbb{Z}}(\mathbb{R})$ is the subgroup generated by the translation $x \mapsto x+1$. Let $\alpha$ (respectively $\alpha'$) be a generator of the center of $G$ (respectively of $G'$). Let $A_{\alpha}= \mathbb{R}- \mathrm{Fix}(\alpha)$ and $A_{\alpha'}=\mathbb{R}- \mathrm{Fix}(\alpha')$.

As the homeomorphisms $\alpha$ and $\alpha'$ commute:

$$\left\{
\begin{array}{l}
\alpha'(A_{\alpha})= A_{\alpha} \\
\alpha(A_{\alpha'})= A_{\alpha'}
\end{array}
\right.
.$$

Take any connected component $I$ of $A_{\alpha}$ and any connected component $I'$ of $A_{\alpha'}$. Then, either $I$ is contained in $I'$, or $I'$ is contained in $I$, or $I$ and $I'$ are disjoint. 

We now prove that only the latter case can occur. Suppose by contradiction that the interval $I$ is strictly contained in the interval $I'$. Let $\sim$ be the equivalence relation defined on $I'$ by
$$ x \sim y \Leftrightarrow (\exists k \in \mathbb{Z}, \ x= \alpha'^{k}(y)).$$
The topological space $I' / \sim$ is homeomorphic to a circle. By Lemma \ref{homeozsurr}, the group $G'$ preserves the interval $I'$. Notice that the group $G'/ < \alpha' > \approx \mathrm{Homeo}_{0}(\mathbb{S}^{1})$ acts on the circle $I' / \sim$. As the group $G'$ commutes with the homeomorphism $\alpha$, this action preserves the nonempty set $(A_{\alpha} \cap I')/ \sim $. As $\alpha'(A_{\alpha})= A_{\alpha}$, the endpoints of the interval $I$ are sent to points in the complement of $A_{\alpha}$ under the iterates of the homeomorphism $\alpha'$. Hence the set $(A_{\alpha} \cap I')/ \sim $ is not equal to the whole circle $I' / \sim$. However, by Theorem 5.3 in \cite{Mats} (see the remark below Theorem \ref{roncirc}), any non-trivial action of the group $\mathrm{Homeo}_{0}(\mathbb{S}^{1})$ on a circle is transitive. Hence, the group $G'/ < \alpha' >$ acts trivially on the circle $I' / \sim$: for any element $\beta'$ of $G'$, and any point $x \in I'$, there exists an integer $k(x, \beta') \in \mathbb{Z}$ such that $\beta'(x)= \alpha'^{k(x,\beta')}(x)$. Fixing such a point $x$, we see that the map
$$\begin{array}{rcl}
G' & \rightarrow & \mathbb{Z} \\
\beta' & \mapsto & k(x,\beta')
\end{array}$$
is a group morphism. Such a group morphism is trivial by Lemma \ref{perfect}. Therefore, the group $G'$ acts trivially on the interval $I'$, a contradiction.

Of course, the case where the interval $I'$ is strictly contained in $I$ is symmetric and cannot occur.

Suppose now that $I=I'$. Take any element $\beta'$ in $G'$. As the homeomorphism $\beta'$ commutes with $\alpha$, by Lemma \ref{homeozsurr}, the homeomorphism $\beta'$ is equal to some element of $G$ on $I$. As the homeomorphism $\beta'$ commute with any element of $G$, there exists a unique integer $k(\beta')$ such that $\beta'_{|I}=\alpha^{k(\beta')}_{|I}.$ The map $k: G \rightarrow \mathbb{Z}$ is a nontrivial group morphism. But such a map cannot exist by Lemma \ref{perfect}. Lemma \ref{commhomeoz} is proved.

It remains to prove Lemma \ref{homeozsurr}.

\begin{proof}[Proof of Lemma \ref{homeozsurr}]
Denote by $t$ a generator of the center of the group $\mathrm{Homeo}_{\mathbb{Z}}(\mathbb{R})$.  

\textbf{Claim 1.} The connected components of the complement of $\mathrm{Fix}(\psi(t))$ are each preserved by the group $\psi(\mathrm{Homeo}_{\mathbb{Z}}(\mathbb{R}))$. Moreover
$$\mathrm{Fix}(\psi(\mathrm{Homeo}_{\mathbb{Z}}(\mathbb{R})))= \mathrm{Fix}(\psi(t)).$$

\textbf{Claim 2.} Any action of the group $\mathrm{Homeo}_{\mathbb{Z}}(\mathbb{R})$ on $\mathbb{R}$ without fixed points is conjugate to the standard action.

It is clear that these two claims imply Lemma \ref{homeozsurr}.

First, let us prove Claim 1. The set $\mathrm{Fix}(\psi(t))$ is preserved under any element in $\psi(\mathrm{Homeo}_{\mathbb{Z}}(\mathbb{R}))$, because any element of this group commutes with the homeomorphism $\psi(t)$. Moreover, any element in $\psi(\mathrm{Homeo}_{\mathbb{Z}}(\mathbb{R}))$ preserves the orientation. Hence any connected component of the complement of $\mathrm{Fix}(\psi(t))$ with infinite length is preserved under the action of the group $\psi(\mathrm{Homeo}_{\mathbb{Z}}(\mathbb{R}))$. We suppose now that any connected component of the complement of $\mathrm{Fix}(\psi(t))$ has a finite length. Let us denote by $\equiv$ the equivalence relation on $\mathbb{R}$ such that $x \equiv y$ if and only if the points $x$ and $y$ belong to the same connected component of $\mathrm{Supp}(\psi(t))$. The morphism $\psi$ induces an action of the group $\mathrm{Homeo}_{\mathbb{Z}}(\mathbb{R}) / <t> \approx \mathrm{Homeo}_{0}(\mathbb{S}^{1})$ on the quotient topological space $\mathbb{R}/ \equiv$ which is homeomorphic to $\mathbb{R}$ if the set $\mathrm{Fix}(\psi(t))$ has a nonempty interior. However, such an action is trivial by the remark below Theorem \ref{roncirc}. Hence, any connected component of $\mathrm{Supp}(\psi(t))$ is preserved by the action. Restricting to one of these connected components if necessary, we can suppose that the closed set $\mathrm{Fix}(\psi(t))$ contains only isolated points. If this set is empty, there is nothing to prove. Otherwise, let
$$\mathrm{Fix}(\psi(t))= \left\{ x_{i}, i \in A \right\},$$
where $A$ is a set contained in $\mathbb{Z}$ and contains $0$ and where the sequence $(x_{i})_{i \in A}$ is strictly increasing. Then, for any element $f$ in $\mathrm{Homeo}_{\mathbb{Z}}(\mathbb{R})$, there exists an integer $i(f)$ such that $\psi(f)(x_{0})=x_{i(f)}$. The map $i: \mathrm{Homeo}_{\mathbb{Z}}(\mathbb{R}) \rightarrow \mathbb{Z}$ is a group morphism: it is trivial by Lemma \ref{perfect}. Claim 1 is proved.

By Claim 1, restricting if necessary the action on a connected component of the complement of $\mathrm{Fix}(\psi(t))$, we can suppose that the homeomorphism $\psi(t)$ has no fixed point. Changing coordinates if necessary, we can suppose that the homeomorphism $\psi(t)$ is the translation $x \mapsto x+1$. The morphism $\psi$ induces an action $\hat{\psi}$ of the group $\mathrm{Homeo}_{\mathbb{Z}}(\mathbb{R})/ <t> \approx \mathrm{Homeo}_{0}(\mathbb{S}^{1})$ on the circle $\mathbb{R}/ \mathbb{Z}$. This action is nontrivial: otherwise, there would exist a nontrivial group morphism $\mathrm{Homeo}_{0}(\mathbb{S}^{1}) \rightarrow \mathbb{Z}$. By the remark below Theorem \ref{roncirc}, there exists a homeomorphism $h$ of the circle $\mathbb{R}/\mathbb{Z}$ such that, for any homeomorphism $f$ in $\mathrm{Homeo}_{\mathbb{Z}}(\mathbb{R})/<t>$ (which can be canonically identified with $\mathrm{Homeo}_{0}(\mathbb{R}/\mathbb{Z})$):
$$ \hat{\psi}(f)=hfh^{-1}.$$
Take a lift $\tilde{h}: \mathbb{R} \rightarrow \mathbb{R}$ of $h$. For any integer $n$, denote by $T_{n}: \mathbb{R} \rightarrow \mathbb{R}$ the translation $x \mapsto x+n$. For any homeomorphism $f$ in $\mathrm{Homeo}_{\mathbb{Z}}(\mathbb{R})$, there exists an integer $n(f)$ such that
$$ \psi(f)=T_{n(f)}\tilde{h}f\tilde{h}^{-1}.$$
However, the map $n: \mathrm{Homeo}_{\mathbb{Z}}(\mathbb{R}) \rightarrow \mathbb{Z}$ is a group morphism: it is trivial by Lemma \ref{perfect}. This completes the proof of Claim 2.
\end{proof}


\begin{thebibliography}{10}

\bibitem{Ban} A. Banyaga, \emph{The structure of classical diffeomorphism groups}, Mathematics and its Applications, Kluwer Academic Publishers Group, Dordrecht, 1997.

\bibitem{Bou} A. Bounemoura \emph{Simplicité des groupes de transformations de surface}, Ensaios Matematicos, vol. 14 (2008), 1-143.

\bibitem{EHN} D. Eisenbud, U. Hirsch, W. Neumann, \emph{Transverse foliations of Seifert bundles and self-homeomorphisms of the circle}, Comment. Math. Helv. 56 (1981), 638-660.

\bibitem{Fis} G.M. Fischer, \emph{On the group of all homeomorphisms of a manifold}, Trans. of the Amer. Math. Soc. 97 (1960), 193-212.

\bibitem{Ghy1} \'E. Ghys, \emph{Groups acting on the circle}, L'Enseignement Mathématique, vol.47 (2001), 329-407.

\bibitem{Ghy2} \'E. Ghys, \emph{Prolongements des difféomorphismes de la sphère}, L'Enseignement Mathématique (2) vol.37 (1991), no. 1-2, 45-59.

\bibitem{Kir} R.C. Kirby, \emph{Stable homeomorphisms and the annulus conjecture}, Ann. of Math. 2nd Ser. 89 (3) (1969), 575-582.

\bibitem{Man} K. Mann, \emph{Homomorphisms between diffeomorphism groups}, preprint available on arXiv:1206.1196.

\bibitem{Mat1} J.N. Mather, \emph{Commutators of diffeomorphisms I}, Comment. Math. Helv. 49 (1974), 512-528.

\bibitem{Mat2} J.N. Mather, \emph{Commutators of diffeomorphisms II}, Comment. Math. Helv. 50 (1975), 33-40.

\bibitem{Mat3} J.N. Mather, \emph{The vanishing of the homology of certain groups of homeomorphisms}, Topology 10 (1971), 297-298.

\bibitem{Mats} S. Matsumoto, \emph{Numerical invariants for semiconjugacy of homeomorphisms of the circle}, Proceedings of the American Mathematical Society, vol. 98 (1986), no 1, 163-168.

\bibitem{MM} S. Matsumoto, S. Morita, \emph{Bounded cohomology of certain groups of homeomorphisms}, Proc. Amer. Math. Soc. 94 (1985), no 3, 539-544.

\bibitem{Mil} E. Militon, \emph{Continuous actions of the group of homeomorphisms of the circle on surfaces}, preprint available on arXiv:1211.0846.

\bibitem{Nav} A. Navas, \emph{Groups of Circle Diffeomorphisms}, Chicago Lectures in Mathematics Series, The University of Chicago Press.

\bibitem{Qui} F. Quinn, \emph{Ends of maps III : Dimensions 4 and 5}, J. Differential Geom. 17 (1982), 503-521.

\end{thebibliography}
\end{document}